\documentclass[12pt,twoside]{amsart}
\usepackage{amsmath}
\usepackage{amsthm}
\usepackage{amsfonts}
\usepackage{amssymb}
\usepackage{latexsym}
\usepackage{mathrsfs}
\usepackage{amsmath}
\usepackage{amsthm}
\usepackage{amsfonts}
\usepackage{amssymb}
\usepackage{latexsym}
\usepackage{geometry}
\usepackage{dsfont}
\usepackage[dvips]{graphicx}
\usepackage{color}
\usepackage[all]{xy}

\date{}
\allowdisplaybreaks[4] \footskip=15pt
\renewcommand{\uppercasenonmath}[1]{}

\usepackage{graphicx,amssymb}
\usepackage[all]{xy}
\usepackage{amsmath}

\allowdisplaybreaks
\usepackage{amsthm}
\usepackage{color}

\theoremstyle{plain}
\newtheorem{theorem}{Theorem}[section]
\newtheorem{proposition}[theorem]{Proposition}
\newtheorem{lemma}[theorem]{Lemma}
\newtheorem{corollary}[theorem]{Corollary}
\newtheorem{example}[theorem]{Example}
\newtheorem*{open question}{Open Question}

\theoremstyle{definition}

\theoremstyle{remark}
\newtheorem{remark}[theorem]{Remark}

\newcommand{\Prufer}{Pr\"{u}fer}

\newcommand{\Id}{\mathrm{Id}}

\def\GV{{\rm GV}}

\def\Hom{{\rm Hom}}
\def\Ext{{\rm Ext}}

\def\Ker{{\rm Ker}}

\def\Im{{\rm Im}}
\def\Coker{{\rm Coker}}

\def\GV{{\rm GV}}

\def\Id{{\rm Id}}

\begin{document}
\begin{center}
{\large  \bf A note on the localizations of generalized injective modules}

\vspace{0.5cm} Xiaolei Zhang$^{a}$

{\footnotesize  $a$.\ School of Mathematics and Statistics, Shandong University of Technology, Zibo 255049, China\\}
\end{center}

\bigskip
\centerline { \bf  Abstract}
\bigskip
\leftskip10truemm \rightskip10truemm \noindent

Let $R$ be a ring and $S$  a multiplicative subset of $R$.  In this note, we study the localization of $S$-injective modules and  $u$-$S$-injective modules under $S$-Noetherian rings and $u$-$S$-Noetherian rings, respectively. The $u$-$S$-absolutely pure property is showed to be preserved under localizations over $S$-coherent rings. Besides, we give an example to show the difference between $S$-injective modules and  $u$-$S$-injective modules, and some counter-examples to deny some questions proposed in \cite{B24}.
\vbox to 0.3cm{}

\leftskip0truemm \rightskip0truemm
\bigskip

\section{Introduction}
Throughout this note, $R$ is always  a commutative ring with identity and $S$ is always a multiplicative subset of $R$, that is, $1\in S$ and $s_1s_2\in S$ for any $s_1\in S$ and any $s_2\in S$.

It is well-known that the localization $E_S$ of any injective $R$-module $E$ is an injective $R_S$-module when $R$ is a Noetherian ring (see \cite[Theorem 4.88]{R09}). This result was extended to several classical rings, such as, polynomial rings with countable variables over Noetherian rings (see \cite[Corollary 17]{D81}), hereditary rings(\cite[Proposition 2]{C06})
, $h$-local \Prufer\  domain(\cite[Corollary 8]{C06}) and \Prufer\  domain of finite character(see \cite[Theorem 10]{C09}). However, this result is not generally  true, see \cite[Theorem 25]{D81} for some polynomial rings with uncountable variables with coefficients,  \cite[Theorem 28]{D81} for some quotient rings of polynomial rings with countable variables, and \cite[Example 1]{C06} for some von Neumann regular rings.

Generalizations of Noetherian rings include $S$-Noetherian rings \cite{ad02}, $u$-$S$-Noetherian rings\cite{QKWCZ21} and  $w$-Noetherian rings \cite{YWZC11}. It is also very interesting to study the localization of the relative injective modules over these generalizations of Noetherian rings. Kim and Wang \cite{KW13} showed that each localization of any
$\GV$-torsion-free injective module is injective over $w$-Noetherian rings. Recently, Baeck \cite{B24} obtain that each localization of any
$S$-injective module is $S$-injective over Noetherian rings, and asked that whether it is true for $S$-Noetherian rings (see \cite[Question 3.11]{B24})? We give a negative answer to this question in this note (see Example \ref{ex-311}). Qi et al.\cite{QKWCZ21} introduced the notions of $u$-$S$-Noetherian rings and $u$-$S$-injective modules. We investigate  the localization of 
$u$-$S$-injective modules  over $u$-$S$-Noetherian rings under some assumptions (see Theorem \ref{usn-loc}). In 1982, Couchot  proved that the each localization of any absolutely pure module is absolutely pure over coherent rings(see \cite[Proposition 1.2]{C82}). Similarly, we obtain that each localization of any $u$-$S$-absolutely pure module is $u$-$S$-absolutely pure over $u$-$S$-coherent rings (see Theorem \ref{usc-loc}). We also give some counter-examples to deny some questions proposed in \cite{B24}.

\section{Preliminaries}

 Recall from  \cite{ad02} that a ring $R$ is called an \emph{$S$-Noetherian ring} if every ideal $I$ of $R$ is $S$-finite with respect to some $s\in S$, that is, there is a finitely generated subideal $K$ of $I$  such that $sI\subseteq K$. Note that, in the definition of $S$-Noetherian rings, the choice of $s\in S$ such that ``$I$ is $S$-finite with respect to  $s\in S$''   is dependent on the ideal $I$. This dependence sets many obstacles to the further study of $S$-Noetherian rings. To overcome this difficulty, Chen et al. \cite{QKWCZ21} introduced the notion of $u$-$S$-Noetherian rings.  A ring $R$ is \emph{$u$-$S$-Noetherian} $($abbreviates uniformly $S$-Noetherian$)$ provided that there exists  $s\in S$ such that for any ideal $I$ of $R$ is $S$-finite with respect to  $s$. Examples of $u$-$S$-Noetherian rings which are not $S$-Noetherian are given in \cite{QKWCZ21}.

 Recall from \cite{z24usc} that an  $R$-module $M$ is called
 \emph{$u$-$S$-finitely presented} (abbreviates \emph{uniformly  $S$-finitely presented}) (with respect to $s$) provided that there is an exact sequence $$0\rightarrow T_1\rightarrow F\xrightarrow{f} M\rightarrow T_2\rightarrow 0$$ with $F$ finitely presented and $sT_1=sT_2=0$.   $R$ is called a
 \emph{$u$-$S$-coherent ring} (abbreviates uniformly  $S$-coherent) (with respect to $s$) provided that there is $s\in S$ such that it is $S$-finite with respect to $s$ and  any finitely generated ideal of $R$ is $u$-$S$-finitely presented with respect to $s$. The difference between coherent rings and $u$-$S$-coherent rings is given in \cite{z24usc}.

The study of $u$-$S$-torsion theories started in 2021 (see \cite{z21}). 
An $R$-module $T$ is called a $u$-$S$-torsion $($abbreviates uniformly $S$-torsion$)$ module  provided that there exists an element $s\in S$ such that $sT=0$.  So  an $R$-module $M$ is  $S$-finite if and only if $M/F$ is $u$-$S$-torsion for some finitely generated submodule $F$ of $M$. An $R$-sequence  $M\xrightarrow{f} N\xrightarrow{g} L$ is called  \emph{$u$-$S$-exact} (at $N$) provided that there is an element $s\in S$ such that $s\Ker(g)\subseteq \Im(f)$ and $s\Im(f)\subseteq \Ker(g)$. We say a long $R$-sequence $\cdots\rightarrow A_{n-1}\xrightarrow{f_n} A_{n}\xrightarrow{f_{n+1}} A_{n+1}\rightarrow\cdots$ is $u$-$S$-exact, if for any $n$ there is an element $s\in S$ such that $s\Ker(f_{n+1})\subseteq \Im(f_n)$ and $s\Im(f_n)\subseteq \Ker(f_{n+1})$. A $u$-$S$-exact sequence $0\rightarrow A\rightarrow B\rightarrow C\rightarrow 0$ is called a short $u$-$S$-exact sequence. A $u$-$S$-short exact sequence $\xi: 0\rightarrow A\xrightarrow{f} B\xrightarrow{g} C\rightarrow 0$ is said to be $u$-$S$-split (with respect to $s$) provided that there is  $s\in S$ and $R$-homomorphism $f':B\rightarrow A$ such that $f'(f(a))=sa$ for any $a\in A$, that is, $f'\circ f=s\Id_A$.
 An $R$-homomorphism $f:M\rightarrow N$ is a \emph{$u$-$S$-monomorphism}  $($resp.,   \emph{$u$-$S$-epimorphism}, \emph{$u$-$S$-isomorphism}$)$ provided $0\rightarrow M\xrightarrow{f} N$   $($resp., $M\xrightarrow{f} N\rightarrow 0$, $0\rightarrow M\xrightarrow{f} N\rightarrow 0$ $)$ is   $u$-$S$-exact.
 It is easy to verify an  $R$-homomorphism $f:M\rightarrow N$ is a $u$-$S$-monomorphism $($resp., $u$-$S$-epimorphism, $u$-$S$-isomorphism$)$ if and only if $\Ker(f)$ $($resp., $\Coker(f)$, both $\Ker(f)$ and $\Coker(f)$$)$ is a  $u$-$S$-torsion module.
\begin{theorem}{\bf ($u$-$S$-analogue of Five Lemma)}\label{s-5-lemma}\cite[Theorem 1.2]{z23}
	Let $R$ be a ring and  $S$ a multiplicative subset of $R$. Consider the following commutative diagram with $u$-$S$-exact rows:
	$$\xymatrix@R=20pt@C=20pt{
		A\ar[d]_{f_A} \ar[r]^{g_1} & B\ar[d]_{f_B}\ar[r]^{g_2} &C\ar[d]^{f_C}\ar[r]^{g_3} &D\ar[r]^{g_4}\ar[d]^{f_D}&E\ar[d]^{f_E} \\
		A'\ar[r]^{h_1} &B'\ar[r]^{h_2}  & C'  \ar[r]^{h_3} & D' \ar[r]^{h_4} & E' .\\}$$
	\begin{enumerate}
		\item  If $f_B$ and $f_D$ are $u$-$S$-monomorphisms and $f_A$ is a $u$-$S$-epimorphism, then $f_C$ is a $u$-$S$-monomorphism.
		\item If $f_B$ and $f_D$ are $u$-$S$-epimorphisms and $f_E$ is a $u$-$S$-monomorphism, then $f_C$ is a $u$-$S$-epimorphism.
		\item If $f_A$ is a $u$-$S$-epimorphism, $f_E$ is a $u$-$S$-monomorphism, and $f_B$ and $f_D$ are $u$-$S$-isomorphisms, then $f_C$ is a $u$-$S$-isomorphism.
		\item  If $f_A$, $f_B$, $f_D$ and $f_E$ are all $u$-$S$-isomorphisms, then $f_C$ is a $u$-$S$-isomorphism.
	\end{enumerate}
\end{theorem}

The notion of $u$-$S$-injective was firstly introduced and studied in \cite{QKWCZ21}.  
An $R$-module $E$ is said to be {\it $u$-$S$-injective} provided that the induced sequence $$0\rightarrow \Hom_R(C,E)\rightarrow \Hom_R(B,E)\rightarrow \Hom_R(A,E)\rightarrow 0$$ is $u$-$S$-exact for any $u$-$S$-exact sequence $0\rightarrow A\rightarrow B\rightarrow C\rightarrow 0$ of $R$-modules.

\begin{theorem}\label{s-inj-ext}\cite[Theorem 4.3]{QKWCZ21} 
	Let $R$ be a ring, $S$ a multiplicative subset of $R$, and $E$ an $R$-module. Then the following conditions are equivalent:
	\begin{enumerate}
		\item  $E$ is  $u$-$S$-injective;
		
		\item for any short exact sequence $0\rightarrow A\xrightarrow{f} B\xrightarrow{g} C\rightarrow 0$ of $R$-modules, the induced sequence $0\rightarrow \Hom_R(C,E)\xrightarrow{g^\ast} \Hom_R(B,E)\xrightarrow{f^\ast} \Hom_R(A,E)\rightarrow 0$ is  $u$-$S$-exact;
		
		\item  $\Ext_R^1(M,E)$ is  uniformly $S$-torsion for any  $R$-module $M$;
		
		\item  $\Ext_R^n(M,E)$ is  uniformly $S$-torsion for any  $R$-module $M$ and any $n\geq 1$.
	\end{enumerate}
\end{theorem}

An $R$-module $M$ is said to be $S$-divisible if $sM=M$ for any $s\in S$. The authors in \cite{QKWCZ21} investigated Baer's Criterion for $S$-divisible modules.

\begin{proposition}\label{s-inj-baer-di}{\bf (Baer's Criterion for $S$-divisible modules)}\cite[Proposition 4.9]{QKWCZ21}	Let $R$ be a ring, $S$ a multiplicative subset of $R$ and $E$ an $S$-divisible  $R$-module. If  there exists $s\in S$ such that  $s\Ext_R^1(R/I,E)=0$ for any ideal $I$ of $R$, then $E$ is $u$-$S$-injective.
\end{proposition}

\section{The localization of $u$-$S$-injective modules over $u$-$S$-Noetherian rings}

Recently, Baeck  \cite{B24} introduced the notion of $S$-injective modules from the experiences of $S$-finite modules. An $R$-module $M$ is said to be $S$-injective if there exists an injective $R$-submodule $Q$ of $M$ such that $Ms\subseteq Q\subseteq M$. It follows from \cite[Theorem 2.3]{B24}
that an $R$-module $M$ is  $S$-injective if and only if $M\cong L\oplus Q$ for some injective $R$-module $Q$ and $u$-$S$-torsion  $R$-module $L$. Since $u$-$S$-injective modules are closed under $u$-$S$-isomorphisms by \cite[Proposition 4.7]{QKWCZ21}, every $S$-injective module is $u$-$S$-injective. The following example shows that every $u$-$S$-injective module is not $S$-injective in general.

\begin{example} Let $D=k[x,y]$ be the polynomial rings with $k$ a field. Let $N=D/xD$. Then $id_D(N)=id_{N}(N)+1=2>1.$
Set $S=\{x^n\mid n\in \mathbb{N}\}$. Let $E(N)$ be the injective envelope of $N$. We claim that $E(N)/N$ is a  $u$-$S$-injective $D$-module but not $S$-injective.

Indeed, since  $E(N)/N$ is $u$-$S$-isomorphic to the injective module $E(N)$,  $E(N)/N$ is certainly $u$-$S$-injective. Now, we claim that $E(N)/N$ is  not $S$-injective. On contrary, assume that   $E(N)/N$ is  $S$-injective. Then $E(N)/N= Q\oplus T$ for some injective module $Q$ and $u$-$S$-torsion module $T$ by  \cite[Theorem 2.3]{B24}. Assume  $x^nT=0$ for some $n\in\mathbb{N}$. Then $$E(N)/N=  x^n(E(N)/N)= x^nQ\oplus x^nT\cong x^nQ=Q,$$ as every injective module and thus its quotient module over an integral domain are divisible.  Hence $E(N)/N$ is injective, contradicting that  $id_D(N)>1$. 	
\end{example}

Let $R$ be a Noetherian ring, and $S,T$ be multiplicative subsets of $R$ such that $0\not\in sT$ for each $s\in S$. Suppose $M$ is an $S$-injective $R$-module. The author in \cite{B24} asserted $M_T$ is also an $S_T$-injective $R_T$-module. Then he asked the following Question:

\textbf{Question1:}\cite[Question 3.11]{B24}  Let $R$ be an $S$-Noetherian ring and $T$ be multiplicative subsets of $R$ such that $0\not\in sT$ for each $s\in S$. Suppose $M$ is an $S$-injective $R$-module. Is $M_T$  an $S_T$-injective $R_T$-module?

Now, we give a negative answer to this question. 

\begin{example}\label{ex-311}
Let $D=k[y,x_\alpha\mid \alpha\in \mathbb{N}\bigsqcup \mathbb{N}^{\mathbb{N}}]$ be a multivariate polynomial domain over  a field $k$. Let  $S=D-\{0\}$ and $T=\{y^n\mid n\in \mathbb{N}\}$. Then there is an injective $D$-module $E$ such that $E_T$ is not injective as $D_T$-module by \cite[Theorem 25]{D81}. Note that $D$ is obviously an $S$-Noetherian ring. We claim that $E_T$ is also not $S_T$-injective as $D_T$-module. Indeed, assume that $E_T$ is an $S_T$-injective $D_T$-module. Then $$E_T\cong Q\oplus L$$  for some injective $D_T$-module $Q$ and $D_T$-module $L$ satisfying $L\frac{s}{t}=0$ for some $\frac{s}{t}\in S_T$.  Since every injective module over an integral domain is divisible, we have the following equivalences: $$E_T\cong (sE)_T \cong \frac{s}{t}E_T\cong \frac{s}{t}Q\oplus \frac{s}{t}L\cong  \frac{s}{t}Q \cong Q.$$ 
Consequently, $E_T$ is an injective $D_T$-module which is a contradiction.
\end{example}

In rest of this section, we mainly investigate the localization of $u$-$S$-injective modules over $u$-$S$-Noetherian ring.

\begin{lemma}\label{S-Noe-usfp} Let $R$ be a $S$-Noetherian ring. Then every ideal of $R$  is $u$-$S$-finitely presented.
\end{lemma}
\begin{proof} Let $I$ be an ideal of $R$. Then there is $s_1\in S$ and a finitely generated ideal $J$ of $R$ such that $s_1I\subseteq J\subseteq I$. Set $T_1=I/J$. Then $s_1T_1=0$. Consider the exact sequence $0\rightarrow K\rightarrow P\rightarrow J\rightarrow 0$, where $P$ is a finitely generated free $R$-module. Then $K$ is $S$-finite. And so there is $s_2\in S$ and a finitely generated $R$-module $L$ such that $s_2K\subseteq L\subseteq K$. Set $T_2=K/L$. Then $s_2T_2=0$. Combining the exact sequences $0\rightarrow T_2\rightarrow P/L\rightarrow J\rightarrow 0$ and $0\rightarrow J\rightarrow I\rightarrow T_1\rightarrow 0$, we have the following exact sequence $$0\rightarrow T_2\rightarrow P/L\rightarrow I\rightarrow T_1\rightarrow 0,$$ where $s_1s_2T_1=s_1s_2T_2=0$ and $P/L$ is finitely presented. So $I$ is $u$-$S$-finitely presented.
\end{proof}

Let $R$ be a ring, $T$ a multiplicative subset of $R$,  and $M,N$ $R$-modules. Then there is a natural homomorphism:
$$\phi_M:\Hom_R(M,N)_T\rightarrow \Hom_{R_T}(M_T,N_T)$$
with $\phi_M(\frac{g}{1})=\widetilde{g}:\frac{i}{1}\mapsto \frac{g(i)}{1}.$

\begin{proposition}\label{supf-usiso} Let $R$ be a ring , $T$ a multiplicative subset of $R$,  and $M,N$ $R$-modules. If $M$ a $u$-$S$-finitely presented $R$-module, then $\phi_M$ is a $u$-$S$-isomorphism.
\end{proposition}
\begin{proof} Since  $M$ is  a $u$-$S$-finitely presented $R$-module, then  there is an exact sequence $$0\rightarrow T_1\rightarrow F\xrightarrow{f} M\rightarrow T_2\rightarrow 0$$ with $F$ finitely presented and $sT_1=sT_2=0$. Consider the following commutative diagrams with exact rows : 
$$\xymatrix@R=20pt@C=10pt{
0 \ar[r]^{} & \Hom_R(T_2,N)_T\ar[d]_{\phi_{T_2}}\ar[r]^{} &\Hom_R(M,N)_T\ar[d]^{\phi_M}\ar[r]^{} &\Hom_R(\Im(f),N)_T\ar[r]^{}\ar[d]^{\phi_{\Im(f)}}&\Ext^1_R(T_2,N)_T\ar[d]^{\phi^1_{T_2}} \\
0\ar[r]^{} &\Hom_{R_T}((T_2)_T,N_T)\ar[r]^{}  &\Hom_{R_T}(M_T,N_T) \ar[r]^{} & \Hom_{R_T}(\Im(f)_T,N_T) \ar[r]^{} & \Ext^1_{R_T}((T_2)_T,N_T),\\}$$
and
$$\xymatrix@R=20pt@C=15pt{
0 \ar[r]^{} & \Hom_R(\Im(f),N)_T\ar[d]_{\phi_{\Im(f)}}\ar[r]^{}&\Hom_R(F,N)_T\ar[r]^{}\ar[d]^{\phi_F} &\Hom_R(T_1,N)_T\ar[d]^{\phi_{T_1}} \\
0\ar[r]^{} &\Hom_{R_T}(\Im(f)_T,N_T)\ar[r]^{}  &\Hom_{R_T}(F_T,N_T) \ar[r]^{} & \Hom_{R_T}((T_1)_T,N_T).\\}$$
By \cite[Lemma 4.2]{QKWCZ21}, $\Ext^n_{R}(U,M)$ are all 
$u$-$S$-torsion modules for all $R$-module $M$, $u$-$S$-torsion module $U$ and all $n\geq 0$. Since  $F$ is  finitely presented, $\phi_F$ is an isomorphism. It follows by \cite[Theorem 1.2]{z23} that $\phi_{\Im(f)}$, and hence $\phi_M$, are $u$-$S$-isomorphisms.
\end{proof}

Recall that an $R$-module $M$ is said to be $S$-torsion-free provided that $sm=0$ with $s\in S$ and $m\in M$ implies that $m=0$.  The authors in \cite{QKWCZ21} investigated Baer's Criterion for $S$-divisible modules. Now, we investigate the Baer's Criterion for $S$-torsion-free modules.

\begin{proposition}\label{s-inj-baer-sf}{\bf (Baer's Criterion for $S$-torsion-free modules)}	Let $R$ be a ring, $S$ a regular multiplicative subset of $R$ and $E$ an $S$-torsion-free  $R$-module. If  there exists $s\in S$ such that  $s\Ext_R^1(R/I,E)=0$ for any ideal $I$ of $R$, then $E$ is $u$-$S$-injective.
\end{proposition}

\begin{proof}
Let $B$ be an $R$-module, $A$ a submodule of $B$, and $s$  a nonzero-divisor in  $S$ satisfying the necessity. Let $f:A\rightarrow E$ be an $R$-homomorphism. Set
\begin{center}
$\Gamma:=\{(C,d) \mid C$ is a submodule of $B$ containing $A$ and $d|_A=sf\}.$
\end{center}
Since $(A,sf)\in \Gamma$, $\Gamma$ is nonempty. Define $(C_1,d_1)< (C_2,d_2)$ if $C_1\subsetneq sC_2$ and $d_2|_{C_1}=d_1$. Then $(\Gamma, <)$ is a partially ordered set. For any chain $(C_j,d_j)$, set $C_0:=\bigcup\limits_{j}C_j$ and define $d_0(c)=d_j(c)$ if $c\in C_j$. Then $(C_0,d_0)$ is the upper bound of the chain $(C_j,d_j)$. By Zorn's Lemma,  there is a maximal element $(C,d)$  in $\Gamma$.
	
We claim that $sB\subseteq C$. Assume on the contrary that $sB\not\subseteq C$ and let $x=sb\in sB \setminus C$ with $b\in B$. Denote $I:=\{r\in R \mid rb\in C\}$. Then $I$ is an ideal of $R$. Let $h:sI\rightarrow E$ satisfying that $h(sr)=d(rb)$. Note that $h$ is well-defined since $s$ is a nonzero-divisor.  Then there is an $R$-homomorphism $g: R\rightarrow E$ such that $g(sr)=sh(sr)=sd(rb)$ for any $sr\in sI$. Since $E$ is $S$-torsion-free, we have $g(r)=h(sr)=d(rb)$ for any $r\in I$. Set $C_1:=C+Rb$ and define $d_1(c+rb)=d(c)+g(r)$, where $c\in C$ and $r\in R$. If  $c+rb=0$, then $r\in I$, and thus $d(c)+g(r)=d(c)+h(sr)=d(c)+d(rb)=d(c+rb)=0$. Hence $d_1$ is a well-defined homomorphism such that $d_1|_A=sf$.  So $(C_1,d_1)\in  \Gamma$. However $sC_1\subsetneq C$, so $(C_1,d_1)> (C,d)$, which contradicts the maximality of $(C,d)$.

Now,we define $d':B\rightarrow E$ to be $d'(b)=d(sb)$. Then $d'(a)=d(sa)=sf(sa)=s^2f(a)$. Hence $E$ is $u$-$S$-injective module with respect to $s^2.$
\end{proof}

\begin{theorem}\label{usn-loc} Let $R$ be a  $u$-$S$-Noetherian ring with $S$  a regular multiplicative subset of $R$ and $T$ a multiplicative subset of $R$ with $0\not\in sT$ for each $s\in S$. Suppose an $R$-module $M$ is $S$-torsion-free or $S$-divisible. If  $M$ is  $u$-$S$-injective, then $M_T$ is a $u$-$S_T$-injective $R_T$-module.
\end{theorem}
\begin{proof} We may assume that $R$ is  $u$-$S$-Noetherian and $M$ is  $u$-$S$-injective both respect  to some $s\in S$.	 Let $J$ be ideal of $R_T$. Then $J=I_T$ for some ideal $I$ of $R$. Consider the following commutative diagram of exact sequences:
$$\xymatrix@R=20pt@C=10pt{
 \Hom_R(R,M)_T\ar[d]_{\phi_{R}}^{\cong}\ar[r]^{} &\Hom_R(I,M)_T\ar[d]^{\phi_I}\ar[r]^{}&\Ext^1_R(R/I,M)_T\ar[r]^{}\ar[d]^{\phi^1_{R/I}}&\Ext^1_R(R,M)_T =0\\
\Hom_{R_T}(R_T,M_T)\ar[r]^{}  &\Hom_{R_T}(I_T,M_T) \ar[r]^{} & \Ext^1_{R_T}(R_T/I_T,M_T) \ar[r]^{} & \Ext^1_{R_T}(R_T,M_T)=0.\\}$$
By assumption $R$ is a  $u$-$S$-Noetherian ring with respect to $s$, so it is also an $S$-Noetherian ring. Then $I$ is $u$-$S$-finitely presented with respect to $s$ by the proof of Lemma \ref{S-Noe-usfp}. Thus, $\phi_I$ is a $u$-$S$-isomorphism with respect to $s$  by Proposition \ref{supf-usiso}. Hence $\phi^1_{R/I}$ is also a $u$-$S$-isomorphism with respect to  $s\in S$  by Theorem \ref{s-5-lemma}. Since $M$ is a $u$-$S$-injective $R$-module, $\Ext^1_R(R/I,M)=0$ is $u$-$S$-torsion $R$-module with respect to $s$. So $\Ext^1_R(R/I,M)_T=0$, and hence $\Ext^1_{R_T}(R_T/I_T,M_T)$ are  $u$-$S_T$-torsion $R_T$-module with respect to some fixed $t\in S_T$. Thus $M_T$ is a $u$-$S_T$-injective $R_T$-module by Proposition \ref{s-inj-baer-di} or Proposition \ref{s-inj-baer-sf}.
\end{proof}

\begin{remark} We don't know whether the Baer's Criterion for $u$-$S$-injective modules  holds in general:
	
Let $R$ be a ring, $S$ a  multiplicative subset of $R$ and $E$ an $R$-module. Suppose there exists $s\in S$ such that  $s\Ext_R^1(R/I,E)=0$ for any ideal $I$ of $R$. Is $E$  $u$-$S$-injective?

So whether the condition ``$M$ is $S$-torsion-free or $S$-divisible'' in Theorem \ref{usn-loc}  can be deleted?

\end{remark}

\section{The localization of $u$-$S$-absolutely pure modules over $u$-$S$-coherent rings}

Let $R$ be a ring and $S$ be a multiplicative subset of $R$. Recall from \cite{z23usabs} that a short $u$-$S$-exact sequence $0\rightarrow A\rightarrow B\rightarrow C\rightarrow 0$ of $R$-modules is said to be  \emph{ $u$-$S$-pure} provided that for any $R$-module $M$, the induced sequence $0\rightarrow M\otimes_RA\rightarrow M\otimes_RB\rightarrow M\otimes_RC\rightarrow 0$ is also $u$-$S$-exact. An $R$-module $E$ is said to be  \emph{$u$-$S$-absolutely pure} (abbreviates uniformly $S$-absolutely pure) provided that any short $u$-$S$-exact sequence $0\rightarrow E\rightarrow B\rightarrow C\rightarrow 0$ beginning with $E$ is $u$-$S$-pure.

\begin{theorem}\label{c-s-abp}\cite[Theorem 3.2]{z23usabs}
	Let $R$ be a ring, $S$ a multiplicative subset of $R$ and $E$ an $R$-module.
	Then the following statements are equivalent:
	\begin{enumerate}
		\item $E$ is $u$-$S$-absolutely pure;
		\item  any short exact sequence $0\rightarrow E\rightarrow B\rightarrow C\rightarrow 0$ beginning with $E$ is $u$-$S$-pure;
		\item  $E$ is a $u$-$S$-pure submodule in every $u$-$S$-injective module containing $E$;
		\item  $E$ is a $u$-$S$-pure submodule in every injective module containing $E$;
		\item  $E$ is a $u$-$S$-pure submodule in its injective envelope;
		\item there exists an element $s\in S$ satisfying that  for any finitely presented $R$-module $N$, $\Ext_R^1(N,E)$ is  $u$-$S$-torsion with respect to $s$;
		\item  there exists an element $s\in S$ satisfying that  if  $P$ is finitely generated projective, $K$ is a finitely generated submodule of $P$ and  $f:K\rightarrow E$ is an  $R$-homomorphism, then there is an $R$-homomorphism $g:P\rightarrow E$ such that $sf=gi$.
	\end{enumerate}
\end{theorem}

Early in 1982, Couchot \cite{C82} proved that the absolutely pure properties over coherent rings persists under localization. Next, we extend it to $u$-$S$-torsion theories.

\begin{theorem}\label{usc-loc} Let $R$ be a  $u$-$S$-coherent ring and $T$ a multiplicative subset of $R$. If $R$-module $N$ is  $u$-$S$-absolutely pure, then $N_T$ is a $u$-$S_T$-absolutely pure $R_T$-module.
\end{theorem}
\begin{proof} We may assume that $R$ is  $u$-$S$-coherent  and $N$ is  $u$-$S$-absolutely pure both with respect to $s\in S$. Let $M_T$ be any finitely presented $R_T$-module with $M$ a finitely generated  $R$-module. Consider the exact sequence $0\rightarrow K\rightarrow F\rightarrow M\rightarrow 0$ where $F$ is a finitely generated free $R$-module. Then $K_T$ is a finitely generated $R_T$-module. So  we may assume $K$ is finitely generated $R$, and hence $M$ is  finitely presented. Since $R$ is a  $u$-$S$-coherent ring (with respect to $s$), then $F$ is a $u$-$S$-coherent $R$-module  (with respect to $s$). Hence  $K$ is  $u$-$S$-finitely presented (with respect to $s$). Consider the following exact sequences with rows exact:
$$\xymatrix@R=20pt@C=10pt{
		\Hom_R(F,N)_T\ar[d]_{\phi_{F}}^{\cong}\ar[r]^{} &\Hom_R(K,N)_T\ar[d]^{\phi_K}\ar[r]^{}&\Ext^1_R(M,N)_T\ar[r]^{}\ar[d]^{\phi^1_{M}}&\Ext^1_R(F,N)_T =0\\
		\Hom_{R_T}(F_T,N_T)\ar[r]^{}  &\Hom_{R_T}(K_T,N_T) \ar[r]^{} & \Ext^1_{R_T}(M_T,N_T) \ar[r]^{} & \Ext^1_{R_T}(F_T,N_T)=0.\\}$$
By proposition \ref{supf-usiso}, $\phi_K$ is a $u$-$S$-isomorphism (with respect to $s$). So $\phi^1_{M}$ is also a $u$-$S$-isomorphism (with respect to $s$) by Theorem \ref{s-5-lemma}. Since  $N$ is  $u$-$S$-absolutely pure (with respect to $s$), $\Ext^1_R(M,N)$  is a $u$-$S$-torsion $R$-module (with respect to $s$) by Theorem \ref{c-s-abp}. Thus 
$\Ext^1_R(M,N)_T$, and hence $\Ext^1_{R_T}(M_T,N_T)$ is a $u$-$S_T$-torsion $R_T$-module (with respect to $\frac{s}{1}$). Consequently, $N_T$ is a $u$-$S_T$-absolutely pure $R_T$-module (with respect to $\frac{s}{1}$) by Theorem \ref{c-s-abp}.
\end{proof}

The rest of this section is to give some counter-examples to some questions proposed in \cite{B24}. The author in \cite{B24} first extended the classical result in \cite[Lemma 3.94]{l99}.
\begin{lemma}\cite[Lemma 5.5]{B24}
	Let $M$ and $N$ be right $R$-modules, and $P$ be an $R$-module such that $_RP_R=u_1R+\cdots+u_k R$ with $Ru_i=u_iR$ for each $i$. If $M_R\subseteq_e N_R$, then $\Hom_R(_RP_R,M_R)\subseteq_e \Hom_R(_RP_RN_R)$. 
\end{lemma}
Then he asked the following question:\\
\textbf{Question 2} \cite[Question 5.6]{B24}
Suppose  $R\subseteq E$ be ring extension and $R$ satisfies $Es'\subseteq u_1R+\cdots+u_kR\subseteq E$  for some $s'\in S$ with $Ru_i=u_iR$ for each $i$. If $M_R\subseteq_e N_R$, is it true that  $\Hom_R(_EE_R,M_R)\subseteq_e \Hom_R(_EE_R,N_R)$. 

Now, we give a counter-example to this question for commutative rings.

\begin{example} Let $R=\mathbb{Z}\times\mathbb{Q}$, $E=\mathbb{Q}\times\mathbb{Q}$ and $S=\{(0,1),(1,1)\}\subseteq R.$ Then $(0,1)E\subseteq R\subseteq E$. Let $M=\mathbb{Z}\times 0$ and $N=\mathbb{Q}\times 0$. Then $M\subseteq_e N$ as $R$-modules. But $0=\Hom_R(_EE_R,M_R)\subseteq \Hom_R(_EE_R,N_R)\cong \mathbb{Q}\times 0$ is not an essential extension as $E$-modules.
\end{example}

The author in \cite{B24} also consider the 
Hom-property of injective modules.

\begin{corollary}\cite[Corollary 5.8(2)]{B24}  Suppose  $R\subseteq E$ be ring extension and $R$ satisfies $E= u_1R+\cdots+u_kR$ with $Ru_i=u_iR$ for each $i$. Set $M_R=L_R\oplus Q_R$.
If $sE=Es$ for each $s\in S$, then  $M_R$ is injective if and only if $\Hom_R(_EE_R,M_R)$ is injective as $E$-module.
\end{corollary}

Then he ask the following question:\\
\textbf{Question 3} \cite[Question 5.9]{B24} In the above corollary, 
can the condition  ``$E= u_1R+\cdots+u_kR$''  be reduced  to  ``$Es'\subseteq u_1R+\cdots+u_kR\subseteq E$  for some $s'\in S$''? 

Certainly, if  $M_R$ is an injective $R$-module, then $\Hom_R(_EE_R,M_R)$ is always injective as an $E$-module (see \cite[Corollary 3.6(B)]{l99}). But the converse is not true in general.

\begin{example} Let $R=\mathbb{Z}\times\mathbb{Q}$, $E=\mathbb{Q}\times\mathbb{Q}$ and $S=\{(0,1),(1,1)\}\subseteq R.$ Then $(0,1)E\subseteq R\subseteq E$. Let $M=\mathbb{Z}\times\mathbb{Q}$. Then $\Hom_R(_EE_R,M_R)=0\times\mathbb{Q}$ is injective $E$-modules. But $M$ is not injective as an $R$-module.
\end{example}

\end{document}